\documentclass[11pt,a4paper]{amsart}
\usepackage{rotating}
\usepackage[all]{xy}
\usepackage[utf8]{inputenc}
\usepackage{amscd,amssymb,amsopn,amsmath,amsthm,graphics,amsfonts,enumerate,verbatim,calc}

\numberwithin{equation}{section}

\newtheorem{theorem}{Theorem}[section]
\newtheorem{lemma}[theorem]{Lemma}
\newtheorem{proposition}[theorem]{Proposition}
\newtheorem{corollary}[theorem]{Corollary}
\newtheorem{notation}[theorem]{Notation}

\theoremstyle{definition}

\theoremstyle{remark}
\newtheorem{remark}[theorem]{Remark}
\newtheorem{fact}[theorem]{Fact}
\newtheorem{example}[theorem]{Example}
\newtheorem{question}[theorem]{Question}
\newtheorem{conjecture}[theorem]{Conjecture}

\newtheorem{acknowledgement}{Acknowledgement}

\newcommand{\edim}{\operatorname{edim}}

\newcommand{\F}{F_*}

\newcommand{\Spec}{\operatorname{Spec}}

\newcommand{\pd}{\operatorname{pd}}
\newcommand{\fd}{\operatorname{fd}}
\newcommand{\gldim}{\operatorname{gl.dim}}
\newcommand{\wdim}{\operatorname{w.dim}}
\newcommand{\Tor}{\operatorname{Tor}}

\newcommand{\Syz}{\operatorname{Syz}}
\newcommand{\grade}{\operatorname{grade}}

\newcommand{\Gpd}{\operatorname{Gpd}}

\newcommand{\id}{\operatorname{id}}

\newcommand{\Gid}{\operatorname{Gid}}

\newcommand{\Ext}{\operatorname{Ext}}

\newcommand{\depth}{\operatorname{depth}}

\newcommand{\m}{\mathfrak m}

\newcommand{\vil}{\operatornamewithlimits{\varinjlim}}

\newcommand{\lo}{\longrightarrow}
\newcommand{\fm}{\mathfrak{m}}

\newcommand{\fn}{\mathfrak{n}}

\begin{document}
	
	\author[M. Asgharzadeh ]{Mohsen Asgharzadeh }
	\date{}
	\title[Perfect closure detects injective dimension ]{Perfect closure detects injective dimension  }
	
	\address{M. Asgharzadeh}
	\email{mohsenasgharzadeh@gmail.com}

	\subjclass[2020]{13D07; 13A35; 13B22}
	\keywords{absolute integral closure perfect closure; injective dimension; projective dimension; Gorenstein injective dimension; global dimension.}
	
	\begin{abstract}
Let $R$ be a complete noetherian local ring of prime characteristic $p$, and let $R^\infty$ denote its perfect closure. We prove that a finitely generated \(R\)-module $N$ has finite injective dimension if and only if $\operatorname{Ext}_R^i(R^\infty, N) = 0$ for all $i > 0$. As a Gorenstein counterpart, we show that finiteness of the Gorenstein injective (or projective) dimension of $R^\infty$ forces $R$ to be Gorenstein, and we relate this to a non-noetherian Cohen factorization of $R \to R^\infty$. Applications include preservation of Gorensteinness under weakly  etale extensions, structural results on F-coherent and weakly F-nilpotent rings, along with the ascent of Frobenius closure of a parameter ideal  to its powers. Assuming
$\operatorname{Ext}^{i}_{R}(\frac{R}{\mathfrak m},R^{\infty})=0$
for some $i>\dim(R)$, we show $R$ is regular. This  has some applications.
We study the rationality problem of Hilbert–Kunz multiplicity by linking it to 
\(R^\infty\). 
Recall that the global dimension can be viewed as a uniform bound on the injective dimension of modules.
Finally, we determine
\(\gldim(R^\infty)\) by presenting a new bound.\end{abstract}
\maketitle
\tableofcontents

\section{Introduction}

A classical theme in commutative algebra is to detect homological invariants of modules over Noetherian rings through vanishing conditions. In characteristic \(p > 0\), the Frobenius map and its iterates provide a powerful tool for such investigations. Let \(R\) be a Noetherian local ring of characteristic \(p\), and let \(\mathrm{F} : R \to R\) denote the Frobenius map given by \(a \mapsto a^p\). For each integer \(e \ge 0\), we write \({}^e R\) for the \(R\)-algebra whose structure is induced by the \(e\)-th iterate \(\mathrm{F}^e : R \to R\).

In a foundational work, Herzog \cite{her} established the following celebrated criterion:

\begin{theorem}[Herzog \cite{her}]
	\label{thm:herzog}
	Let \(R\) be an \(F\)-finite local ring of characteristic \(p\), and let \(M\) be a finitely generated \(R\)-module. If \(\operatorname{Ext}_R^i({}^e R, M) = 0\) for all \(i > 0\) and infinitely many \(e\), then \(\operatorname{id}_R M < \infty\).
\end{theorem}

While elegant, Herzog's theorem requires checking infinitely many modules \({}^e R\)—one for each Frobenius iterate—to conclude finite injective dimension. This raises a natural question: can the entire family \(\{{}^e R\}_{e \ge 0}\) be encoded into a single test module that simultaneously detects finite injective dimension?
The perfect closure of \(R\), defined as the direct limit

\[
R^\infty := \varinjlim \bigl( R \xrightarrow{\mathrm{F}} R \xrightarrow{\mathrm{F}} R \xrightarrow{\mathrm{F}} \cdots \bigr)
= \varinjlim_{e \ge 0} ({}^e R),
\]

is a natural candidate. This \(R\)-algebra is the colimit of all Frobenius twists and has remarkable homological properties. Bhatt and Scholze \cite{BS1} showed that \(R^\infty\) has finite global dimension in a suitable sense, while   \cite{moh} established that every \(R^+\)-module has finite flat dimension over \(R^+\). Here, \(R^+\) denotes the absolute integral closure of an integral domain \(R\), defined as the integral closure of \(R\) inside an algebraic closure of its fraction field.

The main purpose of this note is to prove that the perfect closure \(R^\infty\) indeed serves as a universal test module for finite injective and projective dimensions. Our first main result is the following:

\begin{theorem}
	\label{thm:main1}
	Let \(R\) be a complete local ring of prime characteristic, and let \(N\) be a finitely generated \(R\)-module. Then the following are equivalent:
	\begin{enumerate}
		\item \(\operatorname{Ext}_R^i(R^\infty, N) = 0\) for all \(i > 0\),
		\item \(\operatorname{id}_R(N) < \infty\).
	\end{enumerate}
\end{theorem}
It maybe worth passing to the direct limit is not formally automatic as $\Ext(-,N)$ is not of this style. Indeed, the proof uses some spectral sequence
arguments together with the result of Bhatt-Scholze.
This result has the following corollaries:

\begin{corollary}
	\label{cor:tor}
	Let \((R,\mathfrak m,k)\) be a local ring of prime characteristic, and let \(L\) be an \(\mathfrak m\)-torsion \(R\)-module. Suppose that
	\(
	\operatorname{Tor}_i^R(R^\infty, L) = 0
	\)
	for all \(i > 0\). Then \(\operatorname{pd}_R(L) < \infty\), and consequently \(R\) is Cohen--Macaulay.
\end{corollary}

\begin{corollary}
	\label{cor:gorenstein1}
	Let \(R\) be a complete local ring of prime characteristic. Then \(R\) is Gorenstein if and only if \(\operatorname{Ext}_R^i(R^\infty, R) = 0\) for all \(i > 0\).
\end{corollary}

\begin{corollary}
	\label{cor:gorenstein2}
	Let \(R\) be a complete local ring of prime characteristic. Then \(R\) is Gorenstein if \(R^\infty\) is of finite Gorenstein projective dimension.
\end{corollary}

In order to connect the Gorenstein analogue of Theorem \ref{thm:main1}, our next results center on  interconnected themes:
\begin{enumerate}
	\item detecting Gorensteinness via the Gorenstein injective dimension of \(R^\infty\);
	\item structural consequences of \(F\)-coherent and weak \(F\)-nilpotent;

\end{enumerate}

\medskip

\noindent
(1): Detection of Gorensteinness via two approaches. 
We prove the following detection theorem:

\begin{theorem}
	\label{thm:gid}
	Let \((R,\mathfrak m)\) be a complete local ring of prime characteristic. If
	\(
	\operatorname{Gid}_R(R^\infty) < \infty,
	\)
	then \(R\) is Gorenstein.
\end{theorem}

As an application, we show the following:

\begin{corollary}
	\label{cor:weaklyetale}
	Suppose \(R \to S\) is weakly \'etale, \(S\) is complete, and \(R\) is Gorenstein. Then \(S\) is Gorenstein.
\end{corollary}

Section 3 suggests second, independent approach to the same detection problem—namely, detecting Gorensteinness from \(\operatorname{Gid}_R(R^\infty)\)—is provided by Cohen factorization. For a complete local domain \(R\), we construct a flat extension \(R \to Q = \varinjlim_e Q_e\), where each \(Q_e\) is noetherian and the transition maps are flat when \(R\) contains a regular purely inseparable subextension. This yields a Cohen factorization
\(
R \xrightarrow{\text{flat}} Q \twoheadrightarrow R^\infty,
\) with \(Q\) coherent. The key open question (see Question~\ref{q2.5}) asks whether \(\operatorname{Gid}_Q(Q/I) < \infty\) implies \(\operatorname{id}_Q(Q) < \infty\). A positive answer would imply that \(R\) is Gorenstein via flat descent. Thus, the Cohen factorization approach offers a potentially powerful alternative route that is of independent interest.

\medskip

(2): The main focus of Section~5 is the following question from \cite{Al}:

\begin{quote}
	\emph{If every parameter ideal of \(R\) is Frobenius closed, does it follow that every power of each parameter ideal is Frobenius closed?}
\end{quote}

We partially answer this affirmatively in Proposition \ref{5.4}. This exemplifies a recurring theme of the paper: the perfection \(R^\infty\) is not merely an auxiliary construction but a powerful test object that encodes and transmits homological information.
We also show in Proposition \ref{4.5} that every complete \(F\)-coherent reduced ring is a quotient of a regular ring by a cohomologically complete-intersection ideal. 
We further establish several properties of weak \(F\)-nilpotent and \(F\)-coherent.

In Section 6 we resolve a conjecture \cite[Conjecture 5.19]{Elham} by showing that:
\begin{theorem}\label{18}
	Let $(R,\mathfrak m)$ be a local ring of characteristic $p>0$. 
	Assume that
	\(
	\operatorname{Ext}^{i}_{R}(\frac{R}{\mathfrak m},R^{\infty})=0
	\)
	for some $i>\dim(R)$. Then $R$ is regular.
\end{theorem}

Theorem \ref{18} has the following two applications:

\begin{enumerate}
	\item[(i)]
	\(
	\operatorname{Ext}^{+}_{R}(R^\infty,R^\infty)=0 \Longrightarrow 
	\text{$R$ is regular}.
	\)
	
	\item[(ii)]
	\(
	[\operatorname{Tor}^{R}_{i}(k,R^\infty)=0 \text{ for some } i>0] \Longrightarrow 
	\text{$R$ is regular}.
	\)
\end{enumerate}

The first drops the isolated singularity hypothesis from \cite{Ryo} when working with the minimal perfect closure. 
The second recovers a result of Aberbach--Li \cite{Ab} via a short and novel argument.

In order to compare \(F\)-coherent rings and weakly \(F\)-nilpotent rings further, the following question was suggested to us by Shimomoto, since Cohen–Macaulay \(F\)-coherent rings have rational Hilbert–Kunz multiplicity; see \cite{moh}.

\begin{question}\label{q:hk-nilpotent}(Shimomoto)
	Is the Hilbert–Kunz multiplicity of \(F\)-nilpotent rings rational?
\end{question}
In Section 7, we give a negative answer  to {Question} \ref{q:hk-nilpotent}   by applying the recent   transcendence
theorem of Das--Landsittel--Pham  \cite[Theorem~5.11]{DLP26}  together with the specific section ring used in its proof.
We prove the additional assertions needed here: that this section ring is
$F$-nilpotent at its homogeneous maximal ideal, and that the ideal with
transcendental multiplicity can be made into a maximal ideal by passing to
a finite purely inseparable subring.  This struck the author as strange, since he had thought in the opposite direction.

Recall that the global dimension can be viewed as a uniform bound on the injective dimension of modules. In particular, the following is related to this work:
\begin{question}(\cite[Question 3.5]{desing})
  Let $R$ be a local domain of prime characteristic. What is $\gldim(R^{\infty})$?
\end{question}

We give an embedding-dimension bound and descend it from completion,
using the perfect-algebra results of Bhatt--Scholze \cite{BS1}.

\begin{theorem}\label{thm:perfectdimension}
	Let $(R,\m,k)$ be a Noetherian local ring of characteristic $p>0$, and
	put $e=\edim R$ and $A=R^{\infty}$. Then
$ \gldim A\le e+1.
$ Consequently every $A$-module, and in particular every quotient
	$A/\mathfrak a$ by an arbitrary ideal, has projective dimension at most
	$e+1$.
\end{theorem}

Here, $\edim(R)$ denotes the embedding dimension of 
$R$ i.e., the minimal number of elements needed to generate the maximal ideal of $R$.

 Finally, we suggest that interested readers consult the forthcoming paper \cite{Ryo} for further homological properties of rings and modules in the context of perfect algebras, as well as for mixed-characteristic analogues of the results presented here.
	\maketitle
		\section{Perfect closure detects finite homological dimensions}

Let $(R,\mathfrak m,k)$ be a Noetherian local ring of  prime characteristic $p$.
\begin{notation}	\begin{enumerate}
		\item The notation \(\pd_R(-)\) (resp. \(\id_R(-)\)) stands for the projective (resp. injective) dimension of \((-)\).
	\item Let $(-)^\vee=\operatorname{Hom}_R(-,E_R(k))$ denote the Matlis duality
functor.\end{enumerate}\end{notation}
\begin{fact}(See \cite[4.1]{sim}).
	\label{fact:matlis-duality}
	Let $L$, $L'$, and $N$ be $R$-modules such that $N$ is Noetherian. Since $E_R(k)$ is injective we have following isomorphisms:
	\[
	\operatorname{Ext}^i_R(N, L')^\vee \cong \operatorname{Tor}^R_i(N, L'^\vee), \qquad
	\operatorname{Tor}^R_i(L, L')^\vee \cong \operatorname{Ext}^i_R(L, L'^\vee),
	\]where the second isomorphism is a consequence of Hom–tensor adjointness.
\end{fact}

\begin{proposition}\label{21}
	Let $R$ be a  complete local ring of prime characteristic, and let
	$N$ be  a finitely generated  or a complete $R$-module. Then the following are equivalent:
	
	\begin{enumerate}
		\item
$
		\operatorname{Ext}_R^i(R^\infty,N)=0
	 \text{ for all } i>0,
	$
		
		\item
$
		\operatorname{id}_R(N)<\infty.
$
	\end{enumerate}
\end{proposition}

\begin{proof} 

	$(2)\Rightarrow(1)$.
	It follows by the so called Bass' conjecture (see \cite[9.6.2]{BH}), and   from $
	\operatorname{id}_R(N)<\infty
	$ that $R$ is Cohen--Macaulay. Let $\underline{x}:=x_1,\ldots, x_d$ be a parameter sequence. Then $\underline{x}^{[q]}:=x_1^{[q]},\ldots, x_d^{[q]}$ is an $R$-regular sequence. In other words, $\underline{x}$ is $R^{1/q}$-regular, i.e., ${}^e R$ is Cohen-Macaulay. Taking direct limit,
  $R^\infty$ is a balanced big Cohen--Macaulay module.
	By \cite[Corollary~7.7]{sim} we obtain the claim.
Here is a more elementary argument. Since $R$ is Cohen-Macaulay and complete the canonical module exists.
	Any module of finite injective dimension,  has finite $\omega$-resolution (see \cite[3.3.28(b)]{BH}), hence we may assume
	that
$
	N=\omega_R .
$
Recall that $\underline{x}$ is a parameter sequence.	Then the exact sequence
	\[
	0\longrightarrow
	R 
	\longrightarrow \oplus R_{x_i}
\longrightarrow
	\cdots
	\longrightarrow
	R_{x_1\ldots x_d}
	\longrightarrow
	\omega_R^\vee
	\longrightarrow
	0
	\]
	remains exact after tensoring with $R^\infty$, since $R^\infty$ is balanced
	big Cohen--Macaulay and the homologies are just the local cohomology of $R^\infty$. Hence
$
	\operatorname{Tor}_i^R(R^\infty,\omega_R^\vee)=0.
$
	Therefore, taking a Matlis dual, and using second isomorphism in  Fact \ref{fact:matlis-duality}, we have
$$
0	=\operatorname{Ext}_R^i(R^\infty,\omega_R^{\vee\vee})=
	\operatorname{Ext}_R^i(R^\infty,\widehat{\omega_R})=	\operatorname{Ext}_R^i(R^\infty,\omega_R).
$$

	\noindent
	$(1)\Rightarrow(2)$.
	There is a spectral sequence (see \cite[Theorem 10.74]{Rot}):
	\[
	E_2^{pq}
	=
	\operatorname{Ext}_{R^\infty}^{\,p}
	\!\left(
	L,
	\operatorname{Ext}_R^{\,q}(R^\infty,N)
	\right)
	\Longrightarrow
	\operatorname{Ext}_R^{\,p+q}(L,N),
	\]
	where $L$ is any $R^\infty$-module.
	By assumption (1), the spectral sequence collapses at $q=0$.
	Hence, we lead to the following isomorphism
	\[
	\operatorname{Ext}_{R^\infty}^{\,p}
	(L,\operatorname{Hom}_R(R^\infty,N))
	\cong
	\operatorname{Ext}_R^{\,p}(L,N)
\quad(\ast)	\]
	for every $p$.
		Recall from  Bhatt--Scholze \cite{BS1} that
	\[
	g:=\operatorname{gldim}(R^\infty)<\infty.
	\]

Set $L:=R^\infty/\fm R^\infty$. Then as an $R$-module, we have $L=\oplus R /\fm $, since $\fm L=0$.
	Therefore, the left hand side of $ (\ast)$ is zero for all $p>g$. It follows
	\[
	\operatorname{Ext}_R^{\,p}(L,N)=	\operatorname{Ext}_R^{\,p}(\oplus R /\fm ,N)=\prod \operatorname{Ext}_R^{\,p}(  R /\fm ,N)=0
	\qquad\text{for } p\gg0.
	\] Thus, 	$
  \operatorname{Ext}_R^{\,p}(  R /\fm ,N)=0
$  {for all} $p\gg0$.
 	By Roberts' theorem (see \cite[Theorem 2]{Ro2}), $\mu^n(\mathfrak m,N)$ is nonzero only in the range
$
	\bigl[\operatorname{depth}(N),\,\operatorname{id}(N)\bigr].
$
	Hence
$
	\operatorname{id}(N)<\infty.
$
\end{proof}

\begin{remark}\label{aki1}Adopt the notation of {Proposition} \ref{21}.
	Akiyoshi Sannai informed us that the complete assumption on 	
	\begin{equation*}
		R \emph {  is Gorenstein } \Rightarrow \Ext^{>0}_R(R_{\infty}, R) = 0
	\end{equation*}is really needed, by proposing an example in the style of Nagata.
	The	  issue was flagged by an AI-assisted auditing pipeline for mathematical papers that he was developing. Let \(p\) be a prime, let \(k=\mathbb{F}_p(s_1,s_2,s_3,\ldots)\) be the rational function field over \(\mathbb{F}_p\) in countably many variables, and let \(k^p=\{c^p:c\in k\}\). Let \(R\) be the union of the subrings \(L[[t]]\) of \(k[[t]]\), where \(L\) runs over the fields with \(k^p\subseteq L\subseteq k\) and \([L:k^p]<\infty\); this is Nagata's ring \(k^p[[t]][k]\). Let \(K\) be the fraction field of \(R\), and take \(N=R\). Below, \(\iota_e:R\to R^\infty\) denotes the map from the \(e\)-th term of the direct system \((e=0,1,2,\ldots)\), so that \(\iota_{e+1}(r^p)=\iota_e(r)\), and \(R\) acts on \(R^\infty\) through \(\iota_0\).
	
	\begin{itemize}
		\item \(R\) satisfies the standing convention. Any two such fields \(L\) are contained in their compositum, which is again finite over \(k^p\), so the union is directed and \(R\) is a subring of \(k[[t]]\) of characteristic \(p\); it contains \(k\) and \(t\). Every nonzero \(f\) in \(R\) lies in some \(L[[t]]\), so \(f=t^n u\) with \(u\) a unit of \(L[[t]]\), hence of \(R\). Therefore \(R\) is a discrete valuation ring with uniformizer \(t\), in particular a Noetherian local ring. Moreover, \(R\cap t^n k[[t]]=t^n R\) (if \(t^n y\in L[[t]]\) with \(y\in k[[t]]\), then \(y\in L[[t]]\)) and \(k[t]\subseteq R\), so \(R/t^n R=k[t]/(t^n)\); thus \(R/tR=k\) and the completion of \(R\) is \(k[[t]]\).
		
		\item \(R\) is not complete: the series \(a=\sum_{i\ge1}s_i t^i\) lies in \(k[[t]]\) but not in \(R\), since its coefficients would all lie in one field \(L\) of finite degree over \(k^p\), while \([k^p(s_1,\ldots,s_n):k^p]=p^n\) for every \(n\).
		
		\item {Proposition} \ref{21}\((2)\) holds: \(N=R\) is finitely generated, and \(\mathrm{id}_R(R)=1\) by the injective resolution \(0\to R\to K\to K/R\to0\) (\(K\) and \(K/R\) are divisible modules over the principal ideal domain \(R\), hence injective).
	\end{itemize}

	\begin{enumerate}
	\item[Claim] One has \(\mathrm{Ext}_R^1(R^\infty,R)\ne0\).
		\end{enumerate}

	Indeed, we bring the following goer claim:

	\begin{enumerate}
		\item[(a)] Every \(R\)-linear map \(\psi:k[[t]]\to R\) is zero. Put \(c=\psi(1)\). For \(x\) in \(k[[t]]\) and \(n\ge1\), write \(x=r_n+t^n b_n\) with \(r_n\in k[t]\subseteq R\) and \(b_n\in k[[t]]\); then \(\psi(x)-cx=t^n(\psi(b_n)-c b_n)\) lies in \(t^n k[[t]]\) for every \(n\), so \(\psi(x)=cx\). If \(c\ne0\), write \(c=t^m u\) with \(u\) a unit of \(R\); then \(t^m x=u^{-1}\psi(x)\) lies in \(R\cap t^m k[[t]]=t^m R\) for every \(x\) in \(k[[t]]\), so \(k[[t]]\subseteq R\), contrary to the previous bullet. Hence \(c=0\) and \(\psi=0\).
		
		\item[(b)] \(\mathrm{Hom}_R(R^\infty,R)=0\). For \(g\) in \(k[[t]]\) we have \(g^p\in k^p[[t]]\subseteq R\), so \(\theta(g)=\iota_1(g^p)\) defines a ring homomorphism \(\theta:k[[t]]\to R^\infty\), and \(\theta\) is \(R\)-linear because \(\theta(r)=\iota_1(r^p)=\iota_0(r)\) for \(r\) in \(R\). (Since \(R\) is a domain, \(R^\infty=R^{1/p^\infty}\), and \(\theta\) is the inclusion \(k[[t]]\subseteq R^{1/p}\).) Given an \(R\)-linear map \(\phi:R^\infty\to R\) and \(z\) in \(R^\infty\), the map \(g\mapsto\phi(z\theta(g))\) from \(k[[t]]\) to \(R\) is \(R\)-linear, hence zero by (a); taking \(g=1\) gives \(\phi(z)=0\).
		
		\item[(c)] \(\mathrm{Hom}_R(R^\infty,k)\ne0\). The surjection \(R\to R/tR=k\) commutes with the Frobenius maps, so it induces a ring homomorphism \(R^\infty\to k^\infty\) to the perfect closure of \(k\), which is \(R\)-linear and sends \(1\) to \(1\). Composing it with a \(k\)-linear map \(\lambda:k^\infty\to k\) with \(\lambda(1)=1\) (\(k^\infty\) is a \(k\)-vector space through its first copy of \(k\)) gives a nonzero \(R\)-linear map \(R^\infty\to k\).
		
		\item[(d)] Applying \(\mathrm{Hom}_R(R^\infty,-)\) to the exact sequence \(0\to R\to R\to k\to0\), whose first map is multiplication by \(t\), gives an exact sequence \(\mathrm{Hom}_R(R^\infty,R)\to\mathrm{Hom}_R(R^\infty,k)\to\mathrm{Ext}_R^1(R^\infty,R)\). The first term is \(0\) by (b), so the second map is injective, and \(\mathrm{Ext}_R^1(R^\infty,R)\ne0\) by (c).
	\end{enumerate}
	
	So \(N=R\) satisfies {Proposition} \ref{21}\((2)\) but not {Proposition} \ref{21}\((1)\).  The ring \(R\) is not F-finite, since \([k:k^p]=\infty\), and it is not excellent: \(a^p\) lies in \(R\), while \(a\) is not in \(K\) (every element of \(K\) lies in some \(L((t))\), so the completion \(k((t))\) of \(K\) is not separable over \(K\).
\end{remark}

\begin{remark}
The reverse direction of Proposition~\ref{21} is valid without the "complete" assumption. In what follows, we will use this without explicit reference.
\end{remark}
	
\begin{corollary}
	Let $(R,\mathfrak m,k)$ be a local ring of prime characteristic, and let
	$L$ be an $\fm$-torsion $R$-module. Suppose that
$\operatorname{Tor}_i^R(R^\infty,L)=0$ 
 for all positive  $i$.
	Then
$
	\operatorname{pd}_R(L)<\infty,
$
	and consequently $R$ is Cohen--Macaulay.
\end{corollary}

\begin{proof}
In view of  second isomorphism in  Fact \ref{fact:matlis-duality}, we obtain
	\[
	0=
	\operatorname{Tor}_i^R(R^\infty,L)^\vee
	\cong
	\operatorname{Ext}_R^i(R^\infty,L^\vee).
	\]
	
	Since $L$ is $\fm$-torsion, and by \cite[\S 4.2, Lemma]{sim} $L^\vee$ is complete. Hence, by the previous
	proposition,
$
	\operatorname{id}_R(L^\vee)<\infty.
$
Again by previous fact,
$
	\operatorname{fd}_R(L)<\infty,
$
	and therefore
$
	\operatorname{pd}_R(L)<\infty.
$	Finally, the existence of a nonzero $\fm$-torsion module of finite projective
	dimension implies that $R$ is Cohen--Macaulay (see e.g. \cite[Observation 3.15]{moh5}).
\end{proof}
\begin{corollary}\label{25}
	Let $R$ be a  complete local ring of prime characteristic. Then $R$ is Gorenstein if and only if $\operatorname{Ext}_R^+(R^\infty, R) = 0$.
\end{corollary}

\begin{remark}\label{aki2}
	\begin{enumerate}
	\item
Akiyoshi Sannai informed us that the complete assumption on previous result	 is really needed, see also {Remark} \ref{aki1}.
	\item
 The   if part holds in more general setting than that stated.
\end{enumerate}	

\end{remark}

In what follows, we assume some familiarity with Gorenstein homological algebra, see \cite{Totally}
for a nice survey.
The notation \(\Gpd_R(-)\) (resp. \(\Gid_R(-)\)) stands for the Gorenstein  projective (resp. injective) dimension of \((-)\).
\begin{corollary}
	Let $R$ be a complete local ring of prime characteristic. If $R^\infty$ has finite Gorenstein projective dimension, then $R$ is Gorenstein.
\end{corollary}

\begin{proof}
	Recall that $\operatorname{Gfd}_R(-) \leq \operatorname{Gpd}_R(-)$, and that for every $R$-module $M$ of finite Gorenstein flat dimension, we have the equality
	\[
	\operatorname{Gfd}_R(M) = \sup \left\{ \operatorname{edepth}(R_{\mathfrak p}) - \operatorname{edepth}_{R_{\mathfrak p}}(M_{\mathfrak p}) \mid \mathfrak p \in \operatorname{Spec} R \right\},
	\]where $\operatorname{edepth}$ means $\Ext-depth$.
	Applying this to $M = R^\infty$, and using the fact that $(R^\infty)_{\mathfrak p} \cong (R_{\mathfrak p})^\infty$, we obtain
	\[
	\operatorname{edepth}_{R_{\mathfrak p}}((R^\infty)_{\mathfrak p}) 
	\geq \operatorname{depth}_{R_{\mathfrak p}}((R^\infty)_{\mathfrak p}) 
	\geq \operatorname{depth}_{R_{\mathfrak p}}(R_{\mathfrak p}) 
	= \operatorname{edepth}_{R_{\mathfrak p}}(R_{\mathfrak p}).
	\]
	Hence $\operatorname{Gfd}_R(R^\infty) \leq 0$.
	Moreover, if $\operatorname{Gfd}_R(M)$ is finite, then
	\[
	\operatorname{Gfd}_R(M) = \sup \left\{ i \in \mathbb{Z} \mid \operatorname{Tor}_i^R(E, M) \neq 0 \text{ for some injective } R\text{-module } E \right\}.
	\]
	In particular, $\operatorname{Tor}_i^R(E_R(k), R^\infty) = 0$ for all $i > 0$. Taking the Matlis dual and applying the second isomorphism in Fact \ref{fact:matlis-duality}, we get
	\[
	\operatorname{Ext}_R^i(R^\infty, E_R(k)^\vee) 
	= \operatorname{Ext}_R^i(R^\infty, R) = 0.
	\]
	Then, by Corollary \ref{25}, we conclude that $R$ is Gorenstein.
\end{proof}

	When is $\operatorname{Gid}_R(R^\infty) < \infty$?
A natural answer is precisely in the Gorenstein case. Note that $\operatorname{Gid}_R(R^\infty) \leq \operatorname{id}_R(R^\infty)$, with equality whenever $R^\infty$ has finite injective dimension over $R$ (see \cite[Proposition 3.10]{Totally}). The latter property has been fully investigated   in \cite{Ryo} even in the mixed characteristic. In fact,  $\operatorname{id}_R(R^\infty)$ characterizes the regularity (this was an open problem in the earlier versions of this draft).
 
 \begin{remark}
Suppose $R$ is F-splitting. Then $\operatorname{Gid}_R(R^\infty) < \infty$ if and only if $R$ is Gorenstein. 
 \end{remark}

\begin{proof}Suppose $\operatorname{Gid}_R(R^\infty) < \infty$.
Let $E$ be an injective $R$-module. There is an embedding $\Ext^j_R(E,R)\subseteq   \Ext^j_R(E,R^\infty)=0$ for all $j\gg 0$ (see \cite[Corollary 3.7]{Totally}). Then, $\Ext^j_R(E,R)=0$ for all $j\gg 0$, and so $\operatorname{Gid}_R(R) < \infty$. Since $R$ is projective this implies that $\operatorname{Gid}_R(R)=\id_R(R)$ (see \cite[Proposition 3.11]{Totally}), hence
 $R$ is Gorenstein.
\end{proof}

Against to $\Tor(E,-)$, the functor $\Ext(E,-)$ is not behaved with   respect to the flat base change. This is the main difficulty in working with $\operatorname{Gid}_R(R^\infty)$.
In the next sections we present a systematic study of Gorenstein injective dimension of $R^\infty$.

	\maketitle
	
\section{Cohen's factorization for perfect closure}

Let $k$ be a perfect field, and let $(R,\mathfrak m,k)$ be a Noetherian complete  local domain of
characteristic $p>0$. For each $e\geq 0$, consider the local homomorphism
$
\varphi_e:R\to R^{1/p^e}.
$
Choose a Cohen factorization
\(
R\xrightarrow{\dot\varphi_e}Q_e
\stackrel{\varphi'_e}\twoheadrightarrow\widehat{R^{1/p^e}}={R^{1/p^e}}.\)
Thus $Q_e$ is a complete Noetherian local ring,
$\dot\varphi_e$ is flat, and the closed fiber
$
Q_e/\mathfrak mQ_e
$
is regular.
By the weak functoriality of Cohen's factorizations, the
factorizations may, after suitable choices, be made compatible with
the maps
$
R^{1/p^e}\to R^{1/p^{e+1}}.
$
In particular, we obtain a directed system
\[
Q_0\longrightarrow Q_1\longrightarrow Q_2
\longrightarrow\cdots.
\]
We emphasize that the transition maps in above are not asserted to
be flat, but see {Lemma} \ref {cf} below.  The existence of compatible maps
between suitable Cohen factorizations is the weak functoriality
property established by Nasseh and Sather-Wagstaff \cite{sa}.

\begin{lemma}
	With notation as above, the canonical homomorphism
	$
	R\to  Q=\varinjlim_eQ_e
	$
	is flat.
\end{lemma}

\begin{proof}
	Let
	$
	\zeta:=	0\to M'\to M\to M''\to0
	$
	be an exact sequence of $R$-modules. Since $R\to Q_e$ is flat for
	every $e$, tensoring with $Q_e$ gives an exact sequence
	$
	\zeta\otimes_RQ_e
	.
	$
	Hence
	$
	\varinjlim_e(\zeta\otimes_RQ_e)	
	$
	is exact. Since tensor product commutes with filtered colimits,
	$
	\varinjlim_e(M\otimes_RQ_e)
	\cong
	M\otimes_R\left(\varinjlim_eQ_e\right)
	=M\otimes_RQ,
	$
	and similarly for $M'$ and $M''$. Consequently,
	\(
	\zeta\otimes_RQ
	\)
	is exact. 
\end{proof}

The ring $Q$ is generally non-Noetherian. Nevertheless, it is built
from Noetherian local rings, and this allows one to approximate
finitely presented $Q$-modules by modules over the $Q_e$.
We record \cite{G1} as a standard reference for coherent rings.
\begin{lemma}
	Assume that $Q$ is coherent. Let $N$ be a cyclic finitely presented
	$Q$-module. Then, for some $e$, there exists a cyclic finitely
	presented $Q_e$-module $N_e$ such that
	$
	N\cong N_e\otimes_{Q_e}Q.
	$
\end{lemma}

\begin{proof}
	Since $N$ is cyclic, write
	$
	N\cong Q/I
	$
	for an ideal $I$ of $Q$. Since $N$ is finitely presented and $Q$ is
	coherent, the ideal $I$ is finitely generated. Thus
	$
	I=(x_1,\ldots,x_r).
	$
	As
	$
	Q=\varinjlim_eQ_e,
	$
	each $x_i$ is represented by an element of $Q_e$ for all sufficiently
	large $e$. After increasing $e$, choose elements
	$
	x_{1,e},\ldots,x_{r,e}\in Q_e
	$
	whose images in $Q$ are $x_1,\ldots,x_r$. Put
	$
	I_e=(x_{1,e},\ldots,x_{r,e})\subseteq Q_e
	$
	and
	$
	N_e=Q_e/I_e.
	$
	Then
	$
	I_eQ=I,
	$
	and hence
	$
	N_e\otimes_{Q_e}Q
	\cong Q/I_eQ
	\cong Q/I
	=N.
	$
	Since $Q_e$ is Noetherian, $N_e$ is finitely presented. It is
	obviously cyclic.
\end{proof}

Set
\(
A=\bigcup_{e\geq 0}R^{1/p^e}.
\) We summarize the strategy in the following diagram:
\[
\begin{CD}
R @>>> Q=\varinjlim_eQ_e @> 
\twoheadrightarrow	>> A\\
@. @AAA @AAA\\
@. Q_e @>\twoheadrightarrow>> R^{1/p^e}
\end{CD}
\]
with
$
R\to Q_e$ {flat with regular closed fiber},
and
$
R\to Q
$  flat.
Assume that $B\subseteq R$ is a purely inseparable extension of
Noetherian domains of characteristic $p>0$, and that $B$ is regular.
The first observation is that $R$ and $B$ have the same perfect
closure.

\begin{fact}
	If $B\subseteq R$ is purely inseparable, then, inside a common
	algebraic closure of their fraction fields,
	$
	R^\infty=B^\infty.
	$ Indeed, since $B\subseteq R$, we have
	$
	B^\infty\subseteq R^\infty.
	$
	Conversely, let $x\in R$. Since $R/B$ is purely inseparable, there
	exists $e\geq 0$ such that
	$
	x^{p^e}\in B.
	$
	Hence
	$
	x\in B^{1/p^e}\subseteq B^\infty.
	$
	Thus $R\subseteq B^\infty$, and therefore
	$
	R^\infty\subseteq B^\infty.
	$
	The two inclusions give the assertion.
\end{fact}

Thus, in the purely inseparable situation, we may write
$
A=R^\infty=B^\infty.
$
For each $e\geq0$, consider
$
C_e:=R\otimes_B B^{1/p^e}.
$
There are canonical maps
$
C_e\to C_{e+1}
$
induced by
$
B^{1/p^e}\to B^{1/p^{e+1}}.
$ The regularity of $B$ now becomes crucial.

\begin{lemma}\label{cf}
	If $B$ is regular, then the homomorphism
	\(
	B^{1/p^e}\to B^{1/p^{e+1}}
	\)
	is flat for every $e\geq0$. Consequently,
	\(
	C_e\to C_{e+1}
	\)
	is flat. In particular, we get a Cohen factorizations with coherent rings $\{Q,A\}$:
	\[
	R\stackrel{flat}\longrightarrow Q\twoheadrightarrow A:={R^{\infty}},
	\]
\end{lemma}

\begin{proof}
	The map
	\(
	B^{1/p^e}\to B^{1/p^{e+1}}
	\)
	is obtained from the Frobenius homomorphism on $B^{1/p^e}$.
	Since $B$ is regular, every iterate of Frobenius is flat by Kunz's
	theorem. Hence
	\(
	B^{1/p^e}\to B^{1/p^{e+1}}
	\)
	is flat.
	Moreover,
	\[
	C_{e+1}
	=
	R\otimes_BB^{1/p^{e+1}}
	\cong
	\left(R\otimes_BB^{1/p^e}\right)
	\otimes_{B^{1/p^e}}B^{1/p^{e+1}}
	=
	C_e\otimes_{B^{1/p^e}}B^{1/p^{e+1}}.
	\]
	Thus $C_e\to C_{e+1}$ is a base change of the flat map
	\(
	B^{1/p^e}\to B^{1/p^{e+1}},
	\)
	and hence is flat.
\end{proof}

\begin{question}\label{q2.5}
	If
	\(
	\operatorname{Gid}_Q(Q/I)<\infty ,
	\)
	then is $\id_Q(Q)<\infty$?	\end{question}

The points are that $Q/I$ is cyclic and that $Q$ is not noetherian:

\begin{fact} Suppose $R\to \tilde{Q}$ is a faithfully flat morphism of noetherian rings and \(
	\operatorname{Gid}_{\tilde{Q}}(\tilde{Q}/I)<\infty
	\). By \cite[Theorem 3.22]{Totally},
	$\tilde{Q}$ is Gorenstein. Since $R\to \tilde{Q}$ is flat and positive, it follows that $R$ is Gorenstein.\end{fact}

But if the above question has a positive answer, it follows from $\id_Q(Q)<\infty$ that $\id_R(R)<\infty$, and so $R$ is Gorenstein, as it is noetherian. In the next section we change the viewpoint and prove the desired property directly.

\maketitle

\section{Gorenstein injective dimension detected by perfect closure}

Shimomoto \cite{Shi10} defines \(R\) to be \emph{F-coherent} if the perfect closure \(R^\infty\) is coherent.

\begin{theorem}\label{3.1}
	Suppose \((R,\mathfrak m)\) is an \(F\)-coherent ring. If
	\(
	\operatorname{Gid}_R(R^\infty) < \infty,
	\)
	then \(R\) is Gorenstein.
\end{theorem}

\begin{proof}
	Recall from \cite{BH} that the global dimension of \(R^\infty\) is finite, In particular, we have an \(R^\infty\)-resolution:
	\[
	0 \longrightarrow (R^\infty)^{n_t} \longrightarrow \cdots \longrightarrow (R^\infty)^{n_0} \longrightarrow R^\infty / \mathfrak m R^\infty \longrightarrow 0,
	\]
	where the \(n_i\) are integers. Since \(R \subseteq R^\infty\), the above exact sequence is indeed \(R\)-linear.
	
	We use the following two properties of Gorenstein injective modules (see \cite[Proposition 3.9, Proposition 3.3]{Totally}):
	
	\begin{enumerate}
		\item Finite direct sums of Gorenstein injective modules are Gorenstein injective.
		\item If
		\(
		0 \longrightarrow M_1 \longrightarrow M_2 \longrightarrow M_3 \longrightarrow 0
		\)
		is a short exact sequence and any two of the modules have finite Gorenstein dimension, then so does the third.
	\end{enumerate}
	
	Property (1) implies that any finite direct sum of modules of finite injective dimension has finite Gorenstein injective dimension.
	
	Since \(R^\infty\) has finite Gorenstein injective dimension, the above resolution implies that
	\(
	R^\infty / \mathfrak m R^\infty
	\)
	has finite Gorenstein dimension. Now \(R / \mathfrak m\) is a direct summand of
	\(
	R^\infty / \mathfrak m R^\infty.
	\)
	It follows from \cite[Proposition 3.3]{Totally} that
	\[
	\operatorname{Gid}_R(R / \mathfrak m) < \infty.
	\]
	Therefore, \(R\) is Gorenstein (see \cite[Corollary 3.23]{Totally}).
\end{proof}

The following is the main result of this section.

\begin{theorem}\label{4.2}
	Suppose \((R,\mathfrak m)\) is a complete local ring. Then
	\(
	\operatorname{Gid}_R(R^\infty) < \infty
	\)
	if and only if  \(R\) is Gorenstein.
\end{theorem}

\begin{proof}
Recall from \cite{BS1} that there is an \(R^\infty\)-resolution:
	\[
	0 \longrightarrow (R^\infty)^{n_t} \longrightarrow \cdots \longrightarrow (R^\infty)^{n_0} \longrightarrow R^\infty / \mathfrak m R^\infty \longrightarrow 0,
	\]
	where the \(n_i\in\mathbb{N}\cup\{\infty\}\). Since \(R \subseteq R^\infty\), the above exact sequence is indeed \(R\)-linear.
	We use the following two properties of Gorenstein injective modules:
	
	\begin{enumerate}
		\item Since \(R\) is a homomorphic image of a Gorenstein ring of finite Krull dimension, the class of Gorenstein injective modules is closed under direct limits (see \cite[Theorem 3.16]{Totally}).
		\item If
		\(
		0 \longrightarrow M_1 \longrightarrow M_2 \longrightarrow M_3 \longrightarrow 0
		\)
		is a short exact sequence and any two of the modules have finite Gorenstein dimension, then so does the third.
	\end{enumerate}
	
	Property (1) implies that any direct sum of modules of finite injective dimension has finite Gorenstein injective dimension.
	Since \(R^\infty\) has finite Gorenstein injective dimension, the above resolution implies that
	\(
	R^\infty / \mathfrak m R^\infty
	\)
	has finite Gorenstein dimension. Now \(R / \mathfrak m\) is a direct summand of
	\(
	R^\infty / \mathfrak m R^\infty.
	\)
	It follows from \cite[Proposition 3.3]{Totally} that
	\(
	\operatorname{Gid}_R(R / \mathfrak m) < \infty.
	\)
	Therefore, \(R\) is Gorenstein (see \cite[Corollary 3.23]{Totally}).
 The reverse implication is true without any assumption, and indeed it is well-known (see \cite[Corollary 3.23]{Totally}).
\end{proof}
The following gives a new proof of  \cite[Example 3.17(1)]{Ryo}.

\begin{corollary}\label{4.3}
	Suppose 
	\(
	\operatorname{id}_R(R^\infty) < \infty 	\) (resp. 	\(
	\operatorname{fdim}_R(R^\infty) < \infty
	\)).
Then \(R\) is regular.
\end{corollary}

\begin{proof}
By the argument of Theorem \ref{3.1}, it turns out that   
\(
\operatorname{id}_R(k) < \infty	\) (resp. 	\(
\operatorname{fdim}_R(k) < \infty
\)). 
So, \(R\) is regular.
\end{proof}

Suppose
$
\operatorname{Gid}_R(M \otimes R^\infty) < \infty.
$
One may ask whether \(\operatorname{id}_R(M) < \infty\). This is not the case, as the next result indicates. 

\begin{example}
Let $(R,\fm)$ be a  Gorenstein ring which is not  integral domain, and take \(M := \mathfrak p\) where \(\mathfrak p \in \operatorname{Spec}(R)\). Clearly,
	\(
	\operatorname{Gid}_R(M \otimes R^\infty) < \infty,
	\) as $R$ is  Gorenstein.
	If \(\operatorname{id}(\mathfrak p) < \infty\), then we would have \(\operatorname{pd}(\mathfrak p) < \infty\), as \(R\) is Gorenstein. This would imply that \(R\) is a domain, a contradiction.
\end{example}

An extension \(A \to B\) is called \emph{weakly \'etale} (or \emph{absolutely flat}) if both \(A \to B\) and \(B \otimes_A B \to B\) are flat.

\begin{corollary}\label{4.4}
	Suppose \(R \to S\) is weakly \'etale, \(S\) is complete, and \(R\) is Gorenstein. Then \(S\) is Gorenstein.
\end{corollary}

\begin{proof}
	Since \(R \to S\) is weakly \'etale, we have \(S^{1/p^n} \simeq S \otimes_R R^{1/p^n}\). In particular,
	\(
	S^\infty \simeq S \otimes_R R^\infty.
	\)
	Since \(R\) is Gorenstein, Theorem \ref{4.2} gives \(\operatorname{Gid}_R(R^\infty) < \infty\). This implies
	\[
	\operatorname{Gid}_S(S \otimes_R R^\infty) < \infty,
	\]
	and hence \(\operatorname{Gid}_S(S^\infty) < \infty\). By Theorem \ref{4.2}, \(S\) is Gorenstein.
\end{proof}

\begin{question}\label{q46}
	\begin{enumerate}
		\item[(1)] (See \cite{Ryo}). Suppose
		\(
		\operatorname{Ext}^{+}_{R}(R^\infty,R^\infty)=0.
		\)
		Is \(R\) regular?
		
		\item[(2)] (See \cite[Conjecture 5.19]{Elham}). Suppose
		\(
		\operatorname{Ext}^{\gg 0}_{R}(k,R^\infty)=0.
		\)
		Is \(R\) regular?
	\end{enumerate}
\end{question}

In Section 6 we need the following:

\begin{proposition}\label{47}
	Adopt the above notation. Any positive answer to Question \ref{q46}~(2) yields a positive answer to Question \ref{q46}~(1). In particular,  Question \ref{q46}~(1) is valid over F-pure rings.
\end{proposition}

\begin{proof}
	Suppose \(\operatorname{Ext}^{i}_{R}(R^\infty,R^\infty)=0\) for all \(i>0\), and let \(L\) be an \(R^\infty\)-module. Recall that there is a spectral sequence
	\[
	\operatorname{Ext}^{i}_{R^\infty}
	\left(
	L,\operatorname{Ext}^{j}_{R}(R^\infty,R^\infty)
	\right)
	\Longrightarrow
	\operatorname{Ext}^{i+j}_{R}(L,R^\infty).
	\]
	Since \(\operatorname{Ext}^{i}_{R}(R^\infty,R^\infty)=0\), the spectral sequence degenerates, and this gives
	\[
	\operatorname{Ext}^{i}_{R^\infty}
	\left(
	L,\operatorname{Hom}_{R}(R^\infty,R^\infty)
	\right)
	\simeq
	\operatorname{Ext}^{i}_{R}(L,R^\infty).\quad(\ast)
	\]
	Now we apply this for
	\(
	L:=\frac{R^\infty}{\mathfrak m R^\infty},
	\)
	which is a \(k\)-vector space. Recall that
	\(
	g:=\operatorname{gl.dim}(R^\infty)<\infty
	\)
	(see \cite{BS1}). Thus, the left-hand side of \((\ast)\) vanishes for all \(i>g\). Let us apply this to the right-hand side, and observe that
	\(
	\operatorname{Ext}^{i}_{R}(L,-)=0
	\)
	for all \(i>g\), and therefore
	\[
	\prod \operatorname{Ext}^{i}_{R}
	\left(
	\frac{R}{\mathfrak m},
	R^\infty
	\right)
	=
	\operatorname{Ext}^{i}_{R}
	\left(
	\oplus \frac{R}{\mathfrak m},
	R^\infty
	\right)
	=0.
	\]
	So,
	\(
	\operatorname{Ext}^{i}_{R}
	\left(
\frac{R}{\mathfrak m},
	R^\infty
	\right)
	=0.
	\)
	Since we have assumed the validity of Question \ref{q46}~(2), we deduce that \(R\) is regular. The last claim is now clear, as Question \ref{q46}~(2) is valid over F-pure rings.
\end{proof}

	\maketitle
	
\section{Ascent and descent from (to) $R^\infty$ }

The following was asked in \cite[Question 3.9]{Al}:
\begin{question} Assume each parameter ideal in $R$ is Frobenius closed. Is it true that every power of a parameter ideal is also Frobenius closed?
\end{question}

\begin{notation}Let $I\lhd R$ be an ideal.
	By $I^\star$ (resp. $I^F$) we denote the tight closure (resp. Frobenius closure) of $I$.
\end{notation}
\begin{proposition}\label{42}
	Let $R$ be a equi-dimensional and reduced local ring of prime characteristic $p$ which is quotient of a Cohen-Macaulay ring.
	Assume that $R$ is $F$-coherent. Let
	$
	I=(g_1,\ldots,g_d)
	$
	be an ideal generated by a system of parameters of $R$. If $I$ is
	Frobenius closed, then
	$
	(I^n)^F=I^n
	$
	for every integer $n\geq 1$.
\end{proposition}
\begin{proof}
	By F-coherent we see from \cite[Proposition 3.11]{Shi10} that $I^\star=I^F=I$, and by equi-dimensional assumption it follows that $R$ is F-rational (see \cite[Proposition 10.3.5]{BH}). Then $R$ is Cohen-Macaulay  (see \cite[Corollary 10.3.4]{BH}), as it is quotient of a Cohen-Macaulay ring. So, $I$ is generated 	by regular sequence and the claim is clear by 
	\cite[Proposition 3.10]{Al}.
\end{proof}

Here, we are going to drop the restriction on the length of system of parameters, the equi-dimensional,  reduced, and quotient of a Cohen-Macaulay ring 
from {Proposition} \ref{42}. We record a well-known fact:
\begin{fact}[Frobenius descent]
	Let $I$ be an ideal of $R$. If
	$
	x\in IR^\infty\cap R,
	$
	then there exists an integer $e\geq 0$ such that
	$
	x^{p^e}\in I^{[p^e]}.
	$
In particular, $I^F=IR^\infty\cap R$.
\end{fact}

Let us present a property weaker than F-coherent rings. Recall from \cite{MP} that $R$ is called \emph{weakly F-nilpotent}  if  every element of $H^i_\fm(R)$ is annihilated by some   $\F^n$ for
$i < \dim(R)$.
 Recall from \cite[Exercise 68]{MP}  that over excellent F-nilpotent rings also one has
$I^\star=I^F$, and so similar to F-coherent rings.
\begin{proposition}\label{5.4}
	Let $R$ be a  local ring of prime characteristic $p$.
	Assume that $R$ is weakly F-nilpotent. Let
	$
	I=(g_1,\ldots,g_t)
	$
	be an ideal generated by a part of a system of parameters of $R$. If $I$ is
	Frobenius closed, then
	$
	(I^n)^F=I^n
	$
	for every integer $n\geq 1$.
\end{proposition}
\begin{proof}
	We proceed by induction on $n$. The case $n=1$ follows from the hypothesis
	that $I$ is Frobenius closed.
	Since $R$ is weakly F-nilpotent,   $R^\infty$ is a
	balanced big Cohen--Macaulay $R$-algebra (see \cite[Proposition 12.22]{MP}). 
	Suppose that
	$
	(I^n)^F=I^n
	$
	for some $n\geq 1$, and let
	$
	r\in (I^{n+1})^F.
	$
	Since
	$
	I^{n+1}\subseteq I^n,
	$
	we have
	$
	r\in (I^n)^F=I^n.
	$
	Hence we can write
	$
	r=\sum_{|\alpha|=n}c_\alpha g^\alpha,
	$
	where
	$
	\alpha=(\alpha_1,\ldots,\alpha_t),
	$ and
	$|\alpha|=\alpha_1+\cdots+\alpha_t,
	$
	also
	$
	g^\alpha=g_1^{\alpha_1}\cdots g_t^{\alpha_t}.
	$
	Since
	$
	r\in (I^{n+1})^F,
	$
	there exists a power
	$
	q=p^e
	$
	of $p$ such that
	$
	r^q\in (I^{n+1})^{[q]}.
	$
	Set
	$
	h_i=g_i^q,1\leq i\leq t
	$
	and
	$
	Q=(h_1,\ldots,h_t)=I^{[q]}.
	$
	Since Frobenius powers commute with ordinary powers,
	$
	(I^{n+1})^{[q]}
	(I^{[q]})^{n+1}
	Q^{n+1}.
	$
	Therefore
	$
	r^q
	\sum_{|\alpha|=n}c_\alpha^q h^\alpha
	\in Q^{n+1},
	$
	where
	$
	h^\alpha=h_1^{\alpha_1}\cdots h_t^{\alpha_t}.
	$
	Because $R^\infty$ is balanced big Cohen--Macaulay and
	$g_1,\ldots,g_t$ is part of a system of parameters, the sequence
	$
	g_1,\ldots,g_t
	$
	is $R^\infty$-regular. Hence
	$
	h_1=g_1^q,\ldots,h_t=g_t^q
	$
	is also an $R^\infty$-regular sequence.
	Consider the homogeneous polynomial
	$
	F(Y_1,\ldots,Y_t)
	\sum_{|\alpha|=n}c_\alpha^qY^\alpha
	$
	in
	$
	R^\infty[Y_1,\ldots,Y_t].
	$
	The containment
	$
	F(h_1,\ldots,h_t)\in Q^{n+1}R^\infty
	$
	and the regularity of
	$
	h_1,\ldots,h_t
	$
	imply, by the coefficient criterion for powers of a regular sequence, that
	$
	c_\alpha^q\in QR^\infty
	$
	for every $\alpha$ with
	$
	|\alpha|=n.
	$
	Since
	$
	c_\alpha^q\in R,
	$
	the Frobenius descent lemma, applied to the ideal
	$
	Q=I^{[q]},
	$
	gives an integer $f\geq0$ such that
	$
	(c_\alpha^q)^{p^f}\in Q^{[p^f]}.
	$
	Thus
	$
	c_\alpha^{qp^f}
	\in
	(I^{[q]})^{[p^f]}
	I^{[qp^f]}.
	$
	Since
	$
	qp^f
	$
	is again a power of $p$, this shows that
	$
	c_\alpha\in I^F.
	$
	By the hypothesis
	$
	I^F=I,
	$
	we conclude that
	$
	c_\alpha\in I
	$
	for every $\alpha$. Consequently,
	$
	r
	=
	\sum_{|\alpha|=n}c_\alpha g^\alpha
	\in I^{n+1}.
	$
	Therefore
	$
	(I^{n+1})^F=I^{n+1}.
	$
	The induction is complete.
\end{proof}

There are several similarity between F-coherent and weakly nilpotent\footnote{This also was suggested by \cite{DMP24}.}, namely the big Cohen-Macaulayness of their perfection.
\begin{proposition}
	Let $R$ be a reduced local ring of characteristic $p > 0$. Then $R$ is
	weakly nilpotent if and only if the henselization 
	is so.
\end{proposition}

\begin{proof}
This is parallel to \cite[Proposition 3.18]{Shi10}, so we leave its straightforward modification to the reader.
\end{proof}

An ideal $I$ of regular ring ring $R$ is called  cohomologically complete-intersection if
$$H^i_I(R)=0\Longleftrightarrow i\neq \grade(I,R).$$For example, any ideal generated by regular sequence is of this form. A celebrated theorem of Peskine-Szpiro \cite{PS} says any Cohen-Macaulay ideal in a regular ring of prime characteristic is of this mood. Here, is another sample:
\begin{proposition}\label{4.5} Any complete $F$-coherent reduced ring $A$ is the form
	\(
	A \cong   R/I,
	\)
	where $R$ is regular and $I$ is cohomologically complete-intersection.
\end{proposition}
\begin{proof} By Cohen's structure theorem, $A$ is the form
	\(
	A \cong   R/I,
	\)
	where $R$ is regular. Since $F$-coherent implies that $A^\infty$ is big Cohen--Macaulay (see \cite{Shi10}), we have
	\[
	H_\fm^i(A^\infty)=0 \qquad  \forall i\neq  \dim(A) ,
	\]Then we have
	\(
	H_I^i(R)=0\) for all  \(i\neq \operatorname{ht}(I),
	\) see  $(2)\equiv(6)$ in \cite[Example 4.38]{topo}.
	Therefore $A$ is a cohomologically complete-intersection ring.
\end{proof}

\begin{remark}\label{46}
	The converse of above result is not true. Indeed, there is a hypersurface $A$ which is not F-coherent (see \cite[Example 4.4]{Shi10}).
\end{remark}

We conclude this section with a characterization of weakly \(F\)-nilpotent complete local rings.

\begin{corollary}\label{59}
	A complete local ring is weakly \(F\)-nilpotent if and only if it is cohomologically complete-intersection.
\end{corollary}

\begin{proof}
	This follows from the proof of Proposition \ref{4.5}.
\end{proof}

We refer to Remark \ref{46} and Corollary \ref{59} to see an explicit difference between \(F\)-coherent rings and weakly \(F\)-nilpotent rings.
Here is another.

\begin{proposition}\label{obs:limits}
	Let \(\kappa\) be a limit ordinal such that \(\{(R_\gamma, \mathfrak m_\gamma)\}_{\gamma < \kappa}\) is a filtered direct system of weakly \(F\)-nilpotent rings, with \(\mathfrak m_\gamma R_\delta = \mathfrak m_\delta\) for all \(\gamma < \delta < \kappa\). Then \(\varinjlim_{\gamma} R_\gamma\) is weakly \(F\)-nilpotent.
\end{proposition}

\begin{proof}
	Let \(R := \varinjlim_{\gamma} R_\gamma\). By the main result of \cite{O}, \(R\) is Noetherian. Also, \(R\) is local with unique maximal ideal \(\mathfrak m = \varinjlim_{\gamma} \mathfrak m_\gamma\). Suppose \(\mathfrak m\) is generated by \(\underline{x} := x_1, \ldots, x_t\). Then there exists \(\lambda < \kappa\) such that \(x_i \in R_\gamma\) for all \(i\) and all \(\gamma \ge \lambda\). Moreover, \(\mathfrak m = \mathfrak m_\gamma R\) for all \(\gamma \ge \lambda\).
	In view of \cite[Theorem 15.1]{Mat}, we have
	\[
	\dim (R )\le \dim (R_\gamma) + \dim (R / \mathfrak m_\gamma R).
	\]
	Since \(\mathfrak m = \mathfrak m_\gamma R\), we have \(\dim (R / \mathfrak m_\gamma R) = 0\), and therefore \(\dim (R) \le \dim (R_\gamma)\).
	Let \(i < \dim R\). Now,
	\(
	R^\infty \cong \varinjlim_{\gamma < \kappa} (R_\gamma)^\infty,
	\)
	and so
	\[
	\begin{aligned}
	H_{\mathfrak m}^{i}(R^\infty)
	 \cong H_{\underline{x}R}^{i}(R^\infty) 
	&\cong H_{\underline{x}R}^{i}\left(\varinjlim_{\lambda \le \gamma < \kappa} (R_\gamma)^\infty\right) \\
	&\cong \varinjlim_{\lambda \le \gamma < \kappa} H_{\underline{x}}^{i}(R_\gamma^\infty)\cong \varinjlim_{\gamma < \kappa} H_{\mathfrak m_\gamma}^{i}(R_\gamma^\infty) = 0,
	\end{aligned}
	\]where $H_{\underline{x}}^{i}(-)$ stands for the $i$-th local cohomology with respect to $\underline{x}$.
	Thus \(R^\infty\) is big Cohen--Macaulay, and so \(R\) is weakly \(F\)-nilpotent (see \cite[Proposition 12.22]{MP}).
\end{proof}\begin{proposition}\label{obs:limits-fcoherent}
Let \(\kappa\) be a limit ordinal such that \(\{(R_i, \mathfrak m_i);f_{i,j}:R_i\to R_j\}_{i < \kappa}\) is a filtered direct system of  \(F\)-coherent rings, with \(\mathfrak m_\gamma R_\delta = \mathfrak m_\delta\) for all \(\gamma < \delta < \kappa\) and $f_{i,j}$ is  weakly \'etale. Then \(\varinjlim_{i} R_i\) is \(F\)-coherent.
\end{proposition}

\begin{proof}	Let \(R := \varinjlim_{i} R_\gamma\).
It follows by  change of coefficient (see \cite[page 46]{Mat}) along with the weakly \'etaleness that the following map \[f_{i,j}^\infty:(R_i)^\infty\stackrel{\cong}\lo R_i   \otimes_{R_i} {(R_i)^\infty}\stackrel{f_{i,j}\otimes1}\lo R_j\otimes_{R_i}(R_i)^\infty \stackrel{\cong}\lo(R_j)^\infty,\]
is flat. Since $(R)^\infty=\varinjlim_{i} \{(R_i)^\infty;f_{i,j}^\infty\}$
it follows from \cite[Theorem 2.3.3]{G1} that $R$ is \(F\)-coherent.
\end{proof}

\maketitle

\section{A conjecture on $\mu_i(R^\infty)=0$}

Recall that it was conjectured in  \cite[Conjecture 5.19]{Elham} that:

\begin{conjecture}\label{61}
Suppose \(
\operatorname{Ext}^{i }_{R}(k,R^\infty)=0\) for some $i>\dim(R)$.
Then	$R$ is  regular.

\end{conjecture} 
Mahdavi \cite{Elham} showed that the condition $i > \dim(R)$ is genuinely necessary, and established this over rings with isolated singularities. Also, \cite{MPa} answers this in prime and mixed characteristic F-pure rings.

The following fact is well known.

\begin{fact}\label{comuf}(See \cite[Theorem~2.6.2]{G1}).
	Let $A$ be a commutative ring. Then following are equivalent:
	\begin{enumerate}
		\item[(i)]
		\(A \) is coherent;
		
		\item[(ii)] For all finitely presented  module $P$, any directed system $\{M_i\}$ of $A$-modules and all $n\geq 1$, we have $$\Ext^n_A(P,\vil M_i)\cong\vil\Ext^n_A(P,  M_i).$$
	\end{enumerate}
\end{fact}

In particular, Corollary \ref{4.3} extends as follows.

\begin{theorem}\label{63}
	Let $(R,\mathfrak m)$ be a local ring of characteristic $p>0$. 
	Assume that
	\(
	\operatorname{Ext}^{i}_{R}\left(\frac{R}{\mathfrak m},R^{\infty}\right)=0
	\)
	for some $i>\dim(R)$. Then $R$ is regular.
\end{theorem}

\begin{proof}
Recall from \cite{BH} that the global dimension of \(R^\infty\) is finite.
In particular, we have an \(R^\infty\)-resolution:
\[
0 \longrightarrow \oplus_{X_t}  R^\infty   \stackrel{f_t}\longrightarrow \cdots \longrightarrow \oplus_{X_1} R^\infty  \stackrel{f_1}\longrightarrow R^\infty\stackrel{f_0} \longrightarrow R^\infty / \mathfrak m R^\infty \longrightarrow 0,
\]
where the \(X_i\) are sets non necessarily finite. Since \(R \subseteq R^\infty\), the above exact sequence is indeed \(R\)-linear.
Let $\Syz_j$ denote the $j$-th syzygy in this resolution of $R^\infty / \mathfrak m R^\infty$. Then the exact sequence
\[
0\longrightarrow \oplus_{X_t} R^\infty  
\longrightarrow \oplus_{X_{t-1}} R^\infty 
\longrightarrow \Syz_{\ell-1}
\longrightarrow 0
\]
induces a long exact sequence
\[
\operatorname{Ext}^{i}_{R}
(k,\oplus_{X_{t-1}} R^\infty)
\longrightarrow
\operatorname{Ext}^{i}_{R}(k,\Syz_{\ell-1})  \longrightarrow
\operatorname{Ext}^{i+1}_{R}
(k,\oplus_{X_t} R^\infty).
\]
Because $F_{t-1}:=\oplus_{X_{t-1}} R^\infty $ is free, and recall that direct sum is particularly a directed limit, then we have
\[
\operatorname{Ext}^{i}_{R}(k,F_{t-1})=\vil 	\operatorname{Ext}^{i}_{R}(k,R^\infty) =0
\qquad (\ast),
\] because 
 $R$  is coherent, $i>0$ and $k$
is  finitely presented (see {Fact} \ref{comuf}).
Hence the vanishing of $(\ast)$
implies
\[
\operatorname{Ext}^{i+1}_{R}
(k,\oplus_{X_t} R^\infty)=\vil	\operatorname{Ext}^{i+1}_{R}(k,R^\infty) =0,
\]here we need to assume that $i>\dim(R)$, see \cite[Theorem 1.2]{IYENGAR}.
From this,
\(
\operatorname{Ext}^{i}_{R}(k,\Omega_{\ell-1})=0 .
\)
Thus, for every $j$,
	there is a short exact sequence
	\[
	0\longrightarrow \Syz_{j}
	\longrightarrow F_{j-1}
	\longrightarrow \Syz_{j-1}
	\longrightarrow 0 .
	\]
	Hence by induction the vanishing of
	\(
	\operatorname{Ext}^{i}_{R}(k,\Syz_j)
	=0\)
	implies
	\(
	\operatorname{Ext}^{i+1}_{R}(k,\Syz_{j-1})=0 ,
	\)
	 in particular
	\(
	\operatorname{Ext}^{i}_{R}(k,k)=0 .
	\)
By Roberts' theorem (see \cite[Theorem 2]{Ro2}), we know $\mu^n(\mathfrak m,k)$ is nonzero only in the range
$
\bigl[\operatorname{depth}_R(k),\,\operatorname{id}_R(k)\bigr].	$ Therefore,
	\(
	\operatorname{id}_{R}(k)<\infty .
	\) Hence $R$ is a regular local ring.
\end{proof}

The following drops the isolated singularity from \cite[Theorem 3.15(4)]{Ryo} when we work with minimal perfect closure: 
\begin{corollary}\label{cor:ext-perfect-regular}
Suppose 	\(
	\operatorname{Ext}^{+}_{R}(R^\infty,R^\infty)=0\).
Then 	$R$ is regular.
\end{corollary}

\begin{proof}
	Combine {Proposition} \ref{47} along with the previous theorem.
\end{proof}
It may be worth to mention that  an open problem of Huneke asked when is \(
\operatorname{Tor}^{R}_{1}(k,R^+)=0\)? See \cite{Ab}.
Also, a similar argument as {Theorem} \ref{63} gives a short proof of:

 \begin{corollary}(Aberbach-Li)
	Suppose 	\(
	\operatorname{Tor}^{R}_{i}(k,R^\infty)=0\)  for some  $i>0$.
	Then 	$R$ is regular.
\end{corollary}

The following is an easy approach to  Bhatt-Iyengar-Ma, see \cite{bim}. 
\begin{corollary}\label{63a}
	Suppose 	\(
	\operatorname{Tor}_{+}^{R}(R^\infty,R^\infty)=0\).
	Then 	$R$ is regular.
\end{corollary}

\begin{proof}
Against $\Ext$, tor modules commutes with localization. Then by an easy
induction on Krull dimension, we may assume in addition that $R$ is of isolated singularity, then an argument  similar to \cite[Theorem 3.15(4)]{Ryo} implies	$R$ is regular.
\end{proof}
It is worth noting that many of the results--Corollary \ref{cor:ext-perfect-regular} among them--extend to 
$R^+$ as well. We do not pursue these extensions here, however, since our present focus is exclusively on 
$R^\infty$. Since our results are in prime characteristic, it is natural to ask (see \cite{MPa} and \cite{Ryo} for some progress):
\begin{question}Let $R$ be of characteristic zero.  When is $\operatorname{Gid}_R(R^+) < \infty$?
\end{question}

\section{A question by Shimomoto}
\label{subsec:hk-negative}

For a local ring $(B,\mathfrak b)$ of dimension $d$ and a
$\mathfrak b$-primary ideal $J$, write
\[
e_{\mathrm{HK}}(J,B)
=\lim_{e\to\infty}
\frac{\ell_B(B/J^{[p^e]})}{p^{ed}},
\qquad
e_{\mathrm{HK}}(B)=e_{\mathrm{HK}}(\mathfrak b,B).
\]
The multiplicity of a general primary ideal must be distinguished from
the multiplicity of the ring. The final step of the argument explicitly
addresses this distinction.

\begin{theorem}\label{thm:hk-negative}
	Let $k$ be an uncountable algebraically closed field of characteristic
	$p>2$. There exists a five-dimensional excellent local domain $(A,\fn)$,
	essentially of finite type over $k$ and with residue field $k$, such that
	$A$ is $F$-nilpotent and $e_{\mathrm{HK}}(A)$ is transcendental.
	In particular, the answer to Question~\ref{q:hk-nilpotent} is negative.
\end{theorem}

We first fix the precise nilpotence condition. If $(B,\mathfrak b)$ is a
local domain of characteristic $p$ and dimension $d$, its local
cohomology carries the natural $p$-linear Frobenius action, denoted by $F$.
Set
\[
0^*_{H^d_{\mathfrak b}(B)}
=\{\eta:\text{there is }0\ne c\in B
\text{ with }cF^e(\eta)=0\text{ for every }e\gg0\}.
\]
The ring is $F$-nilpotent if a Frobenius iterate annihilates each
$H^i_{\mathfrak b}(B)$ for $i<d$, and an iterate annihilates
$0^*_{H^d_{\mathfrak b}(B)}$. The latter assertion is
equivalent to
$
0^*_{H^d_{\mathfrak b}(B)}=0^F_{H^d_{\mathfrak b}(B)}.
$

\begin{lemma}\label{lem:hk-degrees}
	Let $M=\bigoplus_{a\in\mathbb Z}M_a$ be a graded $k$-vector space of
	finite dimension, with an additive $p$-linear map $F$ satisfying
	$F(M_a)\subseteq M_{pa}$. If $M_0=0$, then $F^e(M)=0$ for some $e$.
\end{lemma}
\begin{proof}
	If $M=0$ there is nothing to prove. Otherwise choose $b\ge1$ such that
	$M_a=0$ whenever $|a|>b$, and choose $e$ with $p^e>b$.
	For every nonzero degree $a$ in the support,
	$|p^ea|\ge p^e>b$. Thus $F^e(M_a)\subseteq M_{p^ea}=0$.
	The degree-zero part is zero by hypothesis, so the same single exponent
	annihilates all of $M$.
\end{proof}

\begin{lemma}\label{lem:hk-cone}
	Let $k$ be a perfect field of characteristic $p>0$, and let
	$S=\bigoplus_{a\ge0}S_a$ be a normal standard graded $k$-domain with
	$S_0=k$ and dimension $d\ge2$. Put $\fm=S_+$ and $X=\operatorname{Proj}S$.
	Assume that $S$ is regular on $\Spec S\setminus\{\fm\}$ and that
	\[
	H^j(X,\mathcal O_X)=0\qquad(j>0).
	\]
	Then $S_{\fm}$ is $F$-nilpotent.
\end{lemma}
\begin{proof}
	We verify finite length, grading, degree-zero vanishing, and nilpotence
	separately. Write $R=S_{\fm}$.
	Every $H^i_{\fm}(S)$ is $\fm$-torsion. Localization at $\fm$ induces
	an isomorphism
$
		H^i_{\fm}(S)\xrightarrow{\ \cong\ }H^i_{\fm R}(R).
$	Indeed, if $u\notin\fm$, write $u=u_0+u_+$ with $u_0\in k^\times$
	and $u_+\in\fm$. On an element killed by $\fm^n$, multiplication
	by $u$ has inverse given by the finite geometric series
	$u_0^{-1}\sum_{j=0}^{n-1}(-u_+/u_0)^j$.
	The isomorphism is compatible with Frobenius. Thus we may use the
	grading on $H^i_{\fm}(S)$ while checking the local condition on $R$.

	The domain $R$ is equidimensional, is a quotient of a regular local ring,
	and is Cohen--Macaulay on its punctured spectrum. Hence $H^i_{\fm R}(R)$ has finite length for $i<d$.
	Here is the local-duality justification. Choose a polynomial presentation
	$P=k[Z_1,\ldots,Z_t]\twoheadrightarrow S$, and localize at
	$\mathfrak q=(Z_1,\ldots,Z_t)$. Graded local duality identifies the
	graded $k$-dual of $H^i_{\fm}(S)$ with
	\(
	\Ext_P^{t-i}(S,P(-t)).
	\)
	Put $c=t-d$. At a prime of $P$ lying over a prime of $S$ other than
	$\fm$, the quotient is regular, and its codimension in the localized
	polynomial ring is $c$. The codimension is constant because $S$ is a
	domain and polynomial rings are catenary. The Auslander--Buchsbaum
	formula gives projective dimension $c$ at that prime. For $i<d$ one
	has $t-i>c$, so the displayed Ext module vanishes there.
	It is a finitely generated graded module supported only at
	$\mathfrak q$, and therefore has finite length. Its graded dual has
	finite length as well. This proves the assertion, including the
	finite-dimensionality over $k$ that will be used below.
Let
	\(
	T=0^*_{H^d_{\fm R}(R)},
	\)
	viewed as a submodule of $H^d_{\fm}(S)$ using
the previous isomorphism. Choose nonzero degree-one generators
	$x_1,\ldots,x_r$ of $S$ over $k$. For every $j$, the ring $R_{x_j}$
	is regular and hence $F$-rational. Since $R$ is an excellent reduced
	local ring, the standard annihilation theorem for tight closure in top
	local cohomology gives an integer $n_j>0$ with
	\(
	x_j^{n_j}T=0.
	\)
	This is the result recorded in \cite[Theorem~2.2(4)]{Singh00}; importantly,
	its statement does not require $R$ itself to be Cohen--Macaulay.
	The ideal $(x_1^{n_1},\ldots,x_r^{n_r})R$ is $\fm R$-primary, so
	$(\fm R)^nT=0$ for some $n$. Top local cohomology is Artinian, and
	therefore its submodule $T$, annihilated by a power of the maximal
	ideal, has finite length. The residue field of $R$ is $k$, so $T$ is also
	finite-dimensional over $k$.

	For completeness, the grading can be verified directly without choosing
	a homogeneous test element. Let $\eta=\sum_a\eta_a\in T$ be its finite
	homogeneous decomposition. A tight-closure witness in $R$ can be written
	$c/u$ with $0\ne c\in S$ and $u\notin\fm$. By
$
H^i_{\fm}(S)\xrightarrow{\ \cong\ }H^i_{\fm R}(R).
$ and the invertibility of $u$ on local cohomology,
	\(
	cF^e(\eta)=0	\)  for \(e\gg0\)
	already holds in $H^d_{\fm}(S)$. Write $c=\sum_b c_b$ as a finite sum
	of homogeneous elements. For all sufficiently large $e$, the degree
	sets
	\(
	\{p^ea+b:c_b\ne0\}
	\)
	belonging to distinct degrees $a$ occurring in $\eta$ are disjoint:
	the nonzero difference $p^e(a-a')$ eventually exceeds every difference
	between degrees occurring in $c$. Consequently, homogeneous components
	of the last vanishing equation give
	\(
	c_bF^e(\eta_a)=0\)
 {for every } $a,b$  { and every sufficiently large } $e$.
	Choose one nonzero $c_b$. It is a valid witness for each $\eta_a$,
	because $S$ is a domain. Hence $\eta_a\in T$ for each $a$, proving
	that $T$ is graded. Also, a witness $cF^e(\eta)=0$ for $e\gg0$
	is a witness for $F(\eta)$ after shifting $e$ by one. Thus $T$ is
	Frobenius-stable.
	The standard section-ring comparison gives
	\(
	[H^i_{\fm}(S)]_0\cong H^{i-1}(X,\mathcal O_X)	\)
	\(i\ge2\).
	These  vanish by hypothesis. For $i=0,1$, normality and $d\ge2$
	give $\depth (R)\ge2$, so both local cohomology modules vanish.
	In particular,
	\(
	[H^i_{\fm}(S)]_0=0\) for \(i<d\) and
\([T]_0\subseteq[H^d_{\fm}(S)]_0=0.
	\)
The natural Frobenius action on graded local cohomology sends degree
	$a$ to degree $pa$. Each lower local cohomology module and $T$ is
	finite-dimensional, graded, Frobenius-stable, and has zero degree-zero
	part. Lemma~\ref{lem:hk-degrees} gives a uniform annihilating
	iterate on each of them. Taking the maximum of the finitely many
	exponents proves $F$-nilpotence of $R$.
\end{proof}

\begin{remark}\label{rem:hk-cone-regular}
	In the application, regularity of the punctured cone follows from
	smoothness of $X$ and the standard grading. More explicitly, if
	$0\ne x\in S_1$, then
	\(
	S_x\cong(S_x)_0[x,x^{-1}].
	\)
	To see this, a homogeneous element of degree $a$ in $S_x$ is uniquely
	$x^a$ times an element of degree zero; the direct-sum grading proves
	injectivity as well as surjectivity of this Laurent-polynomial
	presentation. The ring $(S_x)_0$ is the coordinate ring of
	$D_+(x)\subseteq X$, so it is regular when $X$ is smooth. The opens
	$D(x)$ for the degree-one generators cover
	$\Spec S\setminus\{S_+\}$. Thus every local ring on that punctured
	spectrum is regular.
\end{remark}

\begin{proposition}\label{prop:hk-enriques-nilpotent}
	Let $Y$ be a smooth classical Enriques surface over an algebraically
	closed field of characteristic $p>2$, and let $E$ be a vector bundle of
	rank three on $Y$. Put $X=\mathbb P_Y(E)$. If an ample line bundle $L$
	on $X$ has a normal full section ring
	\[
	S=R(X,L)=\bigoplus_{a\ge0}H^0(X,L^{\otimes a})
	\]
	generated in degree one, then $S_{S_+}$ is a five-dimensional
	$F$-nilpotent local domain.
\end{proposition}
\begin{proof}
	The projective bundle $X$ is smooth and has dimension $2+(3-1)=4$.
	Its section ring has dimension $5$ and
	$\operatorname{Proj}S\cong X$, with $\mathcal O_X(1)$ corresponding
	to $L$. Also $S_0=H^0(X,\mathcal O_X)=k$.
	For $\pi:X\to Y$, the projective-bundle formula gives
	\(
	\pi_*\mathcal O_X=\mathcal O_Y,\)
	\(R^j\pi_*\mathcal O_X=0
	\) where $j>0$.
	These identities can be checked on an open set trivializing $E$, where
	the morphism is $\mathbb P^2_U\to U$ and the higher cohomology of its
	structure sheaf is zero. Leray then gives
	\(
	H^j(X,\mathcal O_X)\cong H^j(Y,\mathcal O_Y)\) where $j>0$.
	For a classical Enriques surface in odd characteristic,
	$h^1(Y,\mathcal O_Y)=h^2(Y,\mathcal O_Y)=0$. The first is its
	irregularity; the second also follows from Serre duality and the fact
	that $K_Y$ is a nontrivial numerically trivial line bundle of order two.
	Indeed, a nonzero section of $K_Y$ would have an effective numerically
	trivial zero divisor, necessarily zero by intersection with an ample
	divisor, and would trivialize $K_Y$. Higher groups vanish because
	$\dim Y=2$. Therefore $H^j(X,\mathcal O_X)=0$ for every $j>0$.
	Remark~\ref{rem:hk-cone-regular} and Lemma~\ref{lem:hk-cone} now apply.
\end{proof}

\begin{lemma}\label{lem:hk-descent}(See \cite[Theorem~2.17]{DMP24}).
	Let $(A,\fn)\subseteq(B,\mathfrak b)$ be a finite extension of local
	domains of characteristic $p>0$, and suppose $B^{p^{e_0}}\subseteq A$
	for some $e_0$. If $B$ is $F$-nilpotent, then $A$ is $F$-nilpotent.
\end{lemma}
\begin{proof}
	Set $q_0=p^{e_0}$. Integrality gives $\dim A=\dim B=d$.
	The extension is purely inseparable, so the induced map on spectra is
	a bijection; in particular $\sqrt{\fn B}=\mathfrak b$.
	Let $y_1,\ldots,y_s$ generate $\fn$. Use the Cech complexes on these
	elements to compute both sets of local cohomology, using radical
	independence on $B$. Inclusion induces additive maps
	\(
	\alpha_i:H^i_{\fn}(A)\to H^i_{\mathfrak b}(B).
	\)
	The numerator belongs to $A$ by hypothesis. This defines additive
	$q_0$-linear maps
	\(
	\beta_i:H^i_{\mathfrak b}(B)\to H^i_{\fn}(A)
	\)
	which commute with Frobenius and satisfy
$$
	(\ast):\quad	\beta_i\alpha_i=F_A^{e_0},\qquad
		\alpha_i\beta_i=F_B^{e_0}.
$$
	They need not be $A$-linear; only these identities and additivity are used.
	For $i<d$, choose $h_i$ with $F_B^{h_i}H^i_{\mathfrak b}(B)=0$.
	Then
	\(
	F_A^{h_i+e_0}
	=\beta_iF_B^{h_i}\alpha_i=0	\)
	 {on}	\(H^i_{\fn}(A).\)
	For the top condition, if $\eta\in0^*_{H^d_{\fn}(A)}$, a witness
	$0\ne c\in A$ remains nonzero in the domain $B$, and
	\(
	cF_B^e(\alpha_d\eta)
	=\alpha_d(cF_A^e\eta)=0\) where \(e\gg0\).
	Hence $\alpha_d\eta\in0^*_{H^d_{\mathfrak b}(B)}$.
	Choose $h$ annihilating this entire latter submodule. Equation
	$(\ast)$ gives
	\(
	F_A^{h+e_0}\eta=\beta_dF_B^h\alpha_d\eta=0.
	\)
	The exponent is independent of $\eta$. This proves the full
	$F$-nilpotence condition on $A$.
\end{proof}

\begin{lemma}\label{lem:hk-maximal}
	Let $S$ be a standard graded $k$-domain of characteristic $p>0$,
	let $\fm=S_+$, and let
	$I\subseteq S$ be a homogeneous $\fm$-primary ideal. Suppose
	$S_{\fm}$ is $F$-nilpotent. There exists a finitely generated graded
	$k$-subalgebra $A_0\subseteq S$, with $\fn_0=(A_0)_+$, such that
	\(
	A=(A_0)_{\fn_0}\subseteq B=S_{\fm}
	\)
	is a finite purely inseparable extension of local domains,
	$A$ is $F$-nilpotent, both residue fields are $k$, and
$ e_{\mathrm{HK}}(A)
		=\frac{e_{\mathrm{HK}}(IB,B)}
		{[\operatorname{Frac}(S):\operatorname{Frac}(A_0)]}.
$
\end{lemma}
\begin{proof}
	Choose nonzero degree-one algebra generators $x_1,\ldots,x_r$ of $S$
	and homogeneous ideal generators $f_1,\ldots,f_t$ of $I$.
	Since $I$ is $\fm$-primary, $\fm^a\subseteq I$ for some $a$.
	Choose a power $q_0=p^{e_0}\ge a$ and put
	\(
	A_0=k[f_1,\ldots,f_t,x_1^{q_0},\ldots,x_r^{q_0}]\subseteq S.
	\)
	This inherits its positive grading from $S$, and $(A_0)_0=k$.
	Its homogeneous maximal ideal $\fn_0$ is generated, as an $A_0$-ideal,
	by the listed positive-degree generators. They all belong to $I$, and
	the $f_i$ already generate $I$ over $S$. Therefore
$\fn_0S=I.$
Every $x_j$ is integral over $A_0$, since it satisfies
	$T^{q_0}-x_j^{q_0}=0$. As $S=A_0[x_1,\ldots,x_r]$, this proves
	that $S$ is finite over $A_0$.
	Furthermore, the $q_0$th power of every polynomial in the $x_j$ with
	coefficients in $k$ belongs to $k[x_1^{q_0},\ldots,x_r^{q_0}]$.
	Thus $S^{q_0}\subseteq A_0$ and the extension is purely inseparable.
	The unique prime of $S$ over $\fn_0$ is $\fm$. Consequently
	$S\otimes_{A_0}(A_0)_{\fn_0}$ is local and equals $S_{\fm}$.
	For the latter equality one can also check denominators directly:
	if $u\in S\setminus\fm$, then $u^{q_0}\in A_0\setminus\fn_0$,
	so $u^{-1}=u^{q_0-1}/u^{q_0}$ already belongs to that localization.
	The same calculation proves $B^{q_0}\subseteq A$.
	Lemma~\ref{lem:hk-descent} shows that $A$ is $F$-nilpotent.
	The residue fields are both $k$, since the ideals are their
	homogeneous maximal ideals.
Finally use the finite-extension formula for Hilbert--Kunz multiplicity,
	\cite[Theorem~2.7]{WY00}: if $A\subseteq B$ is a finite extension of
	local domains, then for a maximal-primary ideal $J$ of $A$,
	\[
	e_{\mathrm{HK}}(JB,B)
	=\frac{[\operatorname{Frac}(B):\operatorname{Frac}(A)]}
	{[B/\mathfrak b:A/\fn]}
	e_{\mathrm{HK}}(J,A).
	\]
	Take $J=\fn$, note that $\fn B=IB$ and that the residue-field degree
	is one, and rearrange. This gives the claim.
	The extension degree is a finite positive integer (in fact a power of
	$p$), so multiplication or division by it preserves transcendence.
\end{proof}

\begin{proof}[Proof of Theorem~\ref{thm:hk-negative}]
	By \cite[Theorem~5.11 and its proof]{DLP26}, there are a smooth
	classical unnodal Enriques surface $Y$, nef genus-one pencil divisors
	$F_0,F_1,F_2$, and a sufficiently large integer $N$ such that, with
	\[
	\begin{split}
		E&=\mathcal O_Y(F_0)\oplus\mathcal O_Y(F_1)\oplus\mathcal O_Y(F_2),\\
		X&=\mathbb P_Y(E),\quad \pi:X\to Y,\quad
		D=F_0+F_1+F_2,\\
		L_N&=\mathcal O_X(1)\otimes\pi^*\mathcal O_Y(ND),
	\end{split}
	\]
	the full section ring $S=R(X,L_N)$ is a normal domain generated in
	degree one. In this same ring there is a homogeneous $S_+$-primary
	ideal $I$ with $e_{\mathrm{HK}}(I S_{S_+},S_{S_+})$ transcendental.
	The use of the specific section ring in the proof of their theorem
	is essential; its abstract existence statement alone would not verify
	the needed geometric hypotheses.
Proposition~\ref{prop:hk-enriques-nilpotent} shows that $S_{S_+}$
	is $F$-nilpotent and has dimension five.
	Apply Lemma~\ref{lem:hk-maximal} to obtain $(A,\fn)$ with transcendental
	$e_{\mathrm{HK}}(A)$. 
\end{proof}

\section{The global dimension of $R^\infty$}
Question 9.2 of \cite{A} asks when
$\pd_{R^{\infty}}(R^{\infty}/\mathfrak a)$ is finite, where $R$ is
a Noetherian local domain of characteristic $p>0$ and $\mathfrak a$
is an ideal of its perfect closure. The finiteness assertion for
complete local domains is \cite[Corollary 3.9]{{moh}}.
We give an embedding-dimension bound and descend it from completion,
using the perfect-algebra results of Bhatt--Scholze \cite{BS1}.

\begin{theorem}\label{thm:perfectdimension}
	Let $(R,\m,k)$ be a Noetherian local ring of characteristic $p>0$, and
	put $e=\edim R$ and $A=R^{\infty}$. Then
$
		\wdim A\le e,$ and $\gldim A\le e+1.$	Consequently every $A$-module, and in particular every quotient
	$A/\mathfrak a$ by an arbitrary ideal, has projective dimension at most
	$e+1$.
\end{theorem}

\subsection{The global dimension}
We recall the precise perfect-algebra input used in the proof.
\begin{lemma}[Bhatt--Scholze]\label{lem:perfecttor}
	For a diagram $C\leftarrow B\to D$ of perfect $\mathbb F_p$-algebras,
	\[
	\Tor_i^B(C,D)=0\qquad(i>0).
	\]
	In particular, if $B\twoheadrightarrow C$ is a surjection of perfect
	algebras, then $C\otimes_B^{\mathbf L}C\simeq C$.
\end{lemma}
This is \cite[Lemma 11.10]{BS1}, following its Proposition 11.6.
The derived tensor notation
records both the usual tensor product and the vanishing of all higher
Tor groups.

\begin{lemma}\label{lem:perfectionflat}
	Perfection preserves flat ring maps. If $R\to S$ is a faithfully flat
	local map of local $\mathbb F_p$-algebras, then
	$R^{\infty}\to S^{\infty}$ is faithfully flat.
\end{lemma}
\begin{proof}
	Set $A=R^{\infty}$. Index the Frobenius diagrams by $j\ge0$,
	writing $R_j=R$ and $S_j=S$ for their rings at stage $j$.
	The compatible maps $R_j\to S_j$ are flat. Each
$
	T_j=A\otimes_{R_j}S_j
$
	is therefore flat over $A$. The Frobenius transition maps induce
	compatible $A$-algebra maps $T_j\to T_{j+1}$. The universal property
	of the colimit gives
$
	\varinjlim_j T_j\cong S^{\infty}.
$
	Indeed, any coefficient from $A$ is represented in some $R_h$, and
	can be absorbed into $S_h$ after passing to a common later stage.
	A filtered colimit of flat modules is flat.
The perfection of a local ring is quasilocal: a represented element
	is a unit exactly when its representative is outside the maximal
	ideal. A local map remains local under perfection. A flat local map
	of nonzero quasilocal rings is faithfully flat, proving the second
	assertion.
\end{proof}

\begin{lemma}\label{lem:countabledimension}({Osofsky})
	Suppose every ideal of a ring $A$ is countably generated. If
	$\wdim A\le b<\infty$, then $\gldim A\le b+1$.
\end{lemma}
\begin{proof}
	We use two standard dimension facts: a countably presented flat
	module has projective dimension at most one, and
	\[
	\gldim A=\sup_{\mathfrak a\subseteq A}\pd_A(A/\mathfrak a).
	\]
	See \cite[Corollary 2.47 and Theorem 2.17]{Osofsky}; the first fact
	also follows from the countable form of Lazard's theorem and its
	telescope resolution \cite{Lazard}.
	
	Every submodule of a finite free $A$-module is countably generated:
	induct on the rank, projecting onto the last coordinate, whose image
	is an ideal. Every submodule of a countable free module is then
	countably generated by intersecting it with the increasing finite
	free summands. Thus $A/\mathfrak a$ has a resolution by countable
	free modules, with countably presented syzygies.
	Its $b$-th syzygy is flat, by dimension shifting and $\fd_A(A/\mathfrak
	a)\le b$, and hence has projective dimension at most one.
	It follows that $\pd_A(A/\mathfrak a)\le b+1$ for every ideal.
	The displayed cyclic-module criterion completes the proof.
\end{proof}

\begin{proof}[Proof of Theorem~\ref{thm:perfectdimension}]
	First let $P$ be a regular local ring of characteristic $p$ and
	dimension $e$, and put $B=P^{\infty}$. By Kunz's theorem
	\cite{Kunz}, the Frobenius transition maps between the copies $P_j$
	are flat, so $P_j\to B$ is flat. For any $B$-module $M$, define
	\(
	M_j=M\otimes_{P_j}B,
	\)
	where $M$ is viewed as a $P_j$-module by restriction. Since
	$\gldim P_j=e$, flat base change gives $\fd_B M_j\le e$.
	There are compatible multiplication maps $M_j\to M$, and
	\(
	\varinjlim_j M_j\cong M.
	\)
	To check this, represent the finitely many coefficients in a tensor
	at a common stage of $B$ and move them to the first tensor factor.
	Tor commutes with filtered colimits, so $\fd_B M\le e$.
	Hence $\wdim B\le e$.	
	Suppose next that $C$ is a perfect quotient of $B$. For any
	$C$-module $M$, Lemma~\ref{lem:perfecttor} and associativity of
	derived tensor products give
	$$(\ast):\quad	M\otimes_B^{\mathbf L}C
		\simeq M\otimes_C^{\mathbf L}
		(C\otimes_B^{\mathbf L}C)
		\simeq M.
$$
	Take a flat $B$-resolution of $M$ of length at most $e$ and tensor
	it with $C$. Equation~$(\ast)$ proves exactness;
	its terms are flat over $C$. Therefore $\wdim C\le e$.
	
	For the ring $R$ in the statement, Cohen's structure theorem gives
	a surjection
	\[
	P=k[[X_1,\ldots,X_e]]\twoheadrightarrow\widehat R,
	\]
	after choosing a coefficient field of $\widehat R$.
	Perfection preserves this surjection. Thus
	$\widehat R^{\infty}$ is a perfect quotient of
	$P^{\infty}$, and its weak dimension is at most $e$.
	Since $R\to\widehat R$ is faithfully flat and local,
	Lemma~\ref{lem:perfectionflat} gives a faithfully flat map
	$A\to\widehat R^{\infty}$. For arbitrary $A$-modules $M,N$,
	flat base change identifies
	\[
	\Tor_i^A(M,N)\otimes_A\widehat R^{{\infty}}
	\cong
	\Tor_i^{\widehat R^{{\infty}}}
	(M\otimes_A\widehat R^{{\infty}},
	N\otimes_A\widehat R^{{\infty}}).
	\]
	The right side vanishes for $i>e$, and faithful flatness gives
	$\wdim A\le e$.
	
	Finally $A$ is the union of the images of the countably many copies
	of $R$ in its perfection diagram. Each image is a quotient of $R$
	and is Noetherian. For any ideal $\mathfrak a\subseteq A$, its
	intersection with each image is finitely generated there.
	The union of these finite sets generates $\mathfrak a$ over $A$.
	Thus every ideal is countably generated, and
	Lemma~\ref{lem:countabledimension} gives $\gldim A\le e+1$.
\end{proof}

\begin{remark}
	The embedding-dimension estimate is convenient because its proof
	requires only a Cohen presentation. In the complete domain case,
	\cite[Corollary 3.9]{{moh}} gives the different estimate
	$\gldim R^{\infty}\le2\dim R+1$, based on
	\cite[Remark 11.33]{BS1}.
\end{remark}

\subsection{Almost regularity}
\label{sec:almostperfect}
\cite[Question 9.5(ii)]{A} asks whether $R^{\infty}$ is almost
regular for the valuation torsion theory $\mathcal T_v^{\infty}$.
Here $R$ is a Noetherian local domain of characteristic $p>0$.
For a valuation $v$ on $R^{\infty}$, the theory consists of modules
$M$ such that for every $m\in M$ and every $\epsilon>0$ there is
$a\in R^{\infty}$ with $v(a)<\epsilon$ and $am=0$.

Following \cite[Section 9]{A}, almost regularity means that
every module $M$ admit a finite complex of projective modules
$P_\bullet$ with $H_0(P_\bullet)\cong M$ and
$H_i(P_\bullet)\in\mathcal T$ for $i>0$.

\begin{corollary}\label{cor:almostperfect}
	Let $(R,\m)$ be a Noetherian local domain of characteristic $p>0$.
	Then $R^{\infty}$ is almost regular with respect to
	$\mathcal T_v^{\infty}$. More generally, it is almost regular with
	respect to every torsion theory on its module category.
	The complexes can be chosen to be actual projective resolutions of
	length at most $\edim R+1$.
\end{corollary}
\begin{proof}
	For every $R^{\infty}$-module $M$,
	Theorem~\ref{thm:perfectdimension} supplies a projective resolution
	\[
	0\longrightarrow P_b\longrightarrow\cdots\longrightarrow P_0
	\longrightarrow M\longrightarrow0,
	\qquad b\le\edim R+1.
	\]
	Delete the augmentation to $M$. Its zeroth homology is $M$, and
	all its positive homology modules are zero. Every torsion theory
	contains the zero module, so this complex has exactly the required
	properties.
\end{proof}

\begin{acknowledgement}
I thank everyone who helped me in carrying out this work. In particular, I thank Ryo Ishizuka and Kazuma Shimomoto for many valuable comments. Akiyoshi Sannai kindly suggested remarks \ref{aki1}, \ref{aki2}, the error was identified by AI Math Referee, a mathematical auditing AI tool being developed by Akiyoshi Sannai. I learned a lot about {Conjecture} \ref{61} with my coauthors Ishizuka \cite{Ryo}, Mahdavi \cite{Elham}, Patankar \cite{MPa} along with different projects. 
\end{acknowledgement}
\maketitle

\end{document}